\documentclass[12pt]{article}

\usepackage{hyperref}

\usepackage{amsfonts, amsmath, amssymb, amsthm}
\usepackage{a4}
\usepackage{graphicx}
\usepackage{subcaption}

 \usepackage{xcolor}

 \definecolor{Refkey}{RGB}{255,127,0}
 \definecolor{Labelkey}{RGB}{127,0,255}
 \makeatletter 
  \def\SK@refcolor{\color{Refkey}}
  \def\SK@labelcolor{\color{Labelkey}}
 \makeatother

  \definecolor{mdg}{RGB}{0,177,0} 
  \definecolor{mdb}{RGB}{0,0,191}
  \definecolor{mddb}{RGB}{0,0,91}
  \definecolor{mdy}{RGB}{255,69,0} 

\DeclareMathOperator{\dethad}{dh}
\DeclareMathOperator{\characteristic}{char}
\DeclareMathOperator{\pos}{pos}
\DeclareMathOperator{\rank}{rank}

\newtheorem{theorem}{Theorem}
\newtheorem{proposition}{Proposition}

\theoremstyle{remark}

\theoremstyle{definition}
\newtheorem{definition}{Definition}
\newtheorem{example}{Example}

\title{Odd-gon relations and their cohomology}
\author{Igor G. Korepanov}

\date{February--July 2022}

\begin{document}

\sloppy

\maketitle

\begin{flushright}{\it To the memory of Aristophanes Dimakis}\end{flushright}

\medskip

\begin{abstract}
A cohomology theory for ``odd polygon'' relations---algebraic imitations of Pachner moves in dimensions 3, 5, \ldots ---is constructed. Manifold invariants based on polygon relations and nontrivial polygon cocycles are proposed. Example calculation results are presented.
\end{abstract}

\section{Introduction}\label{s:i}

\subsection{Pachner moves}\label{ss:m}

A triangulation of a closed piecewise linear (PL) manifold can be transformed into any other triangulation of the same manifold using a finite sequence of elementary re-buildings---\emph{Pachner moves}~\cite{Pachner,Lickorish}. For a manifold~$M$ of dimension~$\nu$, these moves can be described as follows.

The boundary~$\partial\Delta^{\nu+1}$ of a simplex~$\Delta^{\nu+1}$ of the next dimension consists of $\nu+2$ simplices~$\Delta^{\nu}$. Represent~$\partial\Delta^{\nu+1}$ as the union of two parts, the first containing $\kappa$ simplices~$\Delta^{\nu}$ and called \emph{initial cluster}, and the other containing the remaining $(\nu+2-\kappa)$ simplices~$\Delta^{\nu}$ and called \emph{final cluster}, where $\kappa$ is any of the numbers $1,\ldots,\nu+1$. Suppose that a triangulation of~$M$ contains a part isomorphic to the initial cluster, then ``Pachner move $\kappa \mapsto \nu+2-\kappa$'' consists in replacing this part with the final cluster. This is possible because the two clusters have the same boundary, and either of them can be glued to the rest of the triangulation of~$M$ by this boundary.

There are thus $\nu+1$ kinds of Pachner moves in $\nu$ dimensions:
\begin{equation}\label{kP}
  1 \mapsto \nu+1, \qquad \ldots\,, \qquad \nu+1 \mapsto 1,
\end{equation}
and any two equidistant from the ends of the list~\eqref{kP} are inverse to each other.

In this paper, we will be interested mostly in the case of an \emph{odd} dimension
\begin{equation}\label{nun}
\nu=2n-1, \qquad n=2,3,\ldots \,.
\end{equation}
A special role will be played by the `central' moves $n+1\leftrightarrow n$, although we will consider all other moves as well.

\subsection{Polygon relations}\label{ss:p}

A \emph{polygon relation} is an algebraic imitation of a Pachner move.

Before making this definition more specific, we note that, with algebraic formulas of this kind, there is a hope to construct some kind of topological field theories for piecewise linear manifolds, yielding, in particular, PL manifold invariants.

Besides, polygon relations are beautiful algebraic and combinatorial structures that will definitely find connections and applications in many branches of mathematics. Note especially that they arise within the \emph{higher Tamari orders} approach~\cite{DM-H}.

In this paper, we combine the `higher Tamari' and `topological' approaches. Our terminology will largely be in accordance with the mentioned paper~\cite{DM-H} and especially~\cite{DK} (where we show how the ideas of~\cite{DM-H} can be developed both combinatorially and algebraically). In particular, this implies that we call ``odd polygon'', or simply ``odd-gon'', the relation corresponding to the move $n+1\mapsto n$, in the way we are going to explain. As for all other moves, we treat them within the ``full polygon'' framework of Section~\ref{s:F}.

We number the simplices~$\Delta^{\nu}$ taking part in the mentioned move from~1 through $\nu+2=2n+1$. The initial cluster consists, by definition, of the `odd' simplices $1,3,\ldots,2n+1$, while the final cluster---of the `even' simplices $2,4,\ldots,2n$. Each~$\Delta^{\nu}$ has $2n$~faces (of dimension~$\nu-1$), each of these belongs also to one of the other~$\Delta^{\nu}$. For $\Delta^{\nu}$ number~$q$, the other~$\Delta^{\nu}$, and hence the faces of the $q$-th~$\Delta^{\nu}$, form the ordered set
\begin{equation*}
L_q = \{1,\ldots,\hat q,\ldots,2n+1\},
\end{equation*}
where `hat' means omission. We declare `input' the $n$ faces corresponding to the \emph{odd positions} in~$L_q$, and `output' those corresponding to the even positions in~$L_q$.

\begin{example}\label{x:q}
Let $n=3$, then $2n+1=7$, so this corresponds to the \emph{heptagon} relations. For $q=1$, `inputs' are the common faces of this simplex~$\Delta^5$ with simplices $2,4,6$---we denote these faces $12,14,16$, while `outputs' are faces $13,15,17$. For $q=4$, input faces are $14,34,46$, while output faces are $24,45,47$.
\end{example}

We introduce a set~$X$ whose elements are called \emph{colors}. These colors are assigned to $(\nu-1)$-dimen\-sional faces, and for each~$\Delta^{\nu}$, by definition, the $n$ `output' colors are functions of the $n$ `input' colors. For each of the initial and final cluster, there are $\frac{n(n+1)}{2}$ `global input' faces for the whole cluster, whose colors can be taken arbitrary, and determine the colors of all other faces. Among these, there are $\frac{n(n+1)}{2}$ `global output' faces, that is, such faces that are input for none of the simplices~$\Delta^{\nu}$. Moreover, both the `global input' and `global output' faces are the same for the two clusters.

So, each of the clusters provides dependences of the `global output' colors on `global input' colors, and we say that the odd-gon relation holds if these dependences are the same. Combinatorial details can be found in~\cite[Section~II]{DK} (the fact that only \emph{linear} dependences are considered there does not affect the combinatorics) and in Section~\ref{s:r} below.

The `global input' and `global output' faces form together the common \emph{boundary} of the initial and final clusters. Our odd-gon relation means thus that the dependences between the boundary colors are the same for the two clusters.

\subsection{Notations}\label{ss:nota}

$\Delta^m$ denotes an $m$-dimen\-sional simplex, that is, a simplex with $(m+1)$ vertices. Four kinds of simplices will be of most importance in this paper:
\begin{itemize}
 \item $\Delta^{2n}$ ---its boundary~$\partial \Delta^{2n}$ is the union of the l.h.s.\ and r.h.s.\ of Pachner move $(n+1)\mapsto n$ or any other move in $2n-1$ dimensions. We call such simplex \textbf{P-simplex}, from the word `Pachner';
 \item $\Delta^{2n-1}$ ---a simplex of the maximal dimension in a $(2n-1)$-dimen\-sional PL manifold. We call such simplex \textbf{d-simplex}, from the word `dimension';
 \item $\Delta^{2n-2}$ ---a face of~$\Delta^{2n-1}$. These faces will be ``colored'' by elements~$x$ of some ``set of colors''~$X$. We call such simplex simply \textbf{face}, if there is no risk of confusion;
 \item $\Delta^{n-2}$ ---to such simplices special `generating' colorings will correspond. We call such simplex \textbf{g-simplex}, from the word `generating'.
\end{itemize}

For brevity, we often use the following ``complemental'' notations. A d-simplex~$\Delta^{2n-1}$ considered in the context of a Pachner move is determined by pointing out the \emph{only} vertex~$q$ from the set $\{1,\ldots,2n+1\}$ that is does \emph{not} contain. So, in this paper it can be denoted simply as~$q$. Similarly, a face~$\Delta^{2n-2}$ not containing vertices $q$ and~$r$ can be denoted simply~$qr$ or, equivalently,~$rq$; this is of course exactly the face common for d-simplices $q$ and~$r$.

\subsection{Contents of the rest of this paper}\label{ss:r}

Below,
\begin{itemize}\itemsep 0pt
 \item in Section~\ref{s:r}, we construct a specific algebraic odd-gon relation. We start with introducing objects related to simplices of different dimensions, and arrive at a relation between products of matrices, linking together the `combinatorial topological' and `higher Tamari' approaches;
 \item in Section~\ref{s:F}, we extend the mentioned relation from move $(n+1)\mapsto n$ onto all Pachner moves---what we call ``full polygon'';
 \item in Section~\ref{s:coho}, we describe a cohomology theory for our polygon relations;
 \item in Section~\ref{s:f}, we explain how that cohomology theory can be used for constructing invariants of piecewise linear manifolds, using nontrivial polygon $(2n-1)$-cocycles;
 \item in Section~\ref{s:2n-2}, we show one way of how to construct actual manifold invariants. Namely, there exists a $(2n-2)$-cocycle in characteristic zero from which a desired $(2n-1)$-cocycle can be obtained in a finite characteristic (using a procedure reminiscent remotely of a Bockstein homomorphism);
 \item in Section~\ref{s:5c}, we show the existence and present calculations related to a nontrivial heptagon $(2n-1)$-cocycle in characteristic zero. It probably does not have immediate topological applications but demonstrates an interesting algebraic structure;
 \item in Section~\ref{s:d}, we discuss the results and directions of future research.
\end{itemize}

\section{Construction of odd-gon relations}\label{s:r}

\subsection[Definition of $(n-2)$-simplex vectors]{Definition of $\boldsymbol{(n-2)}$-simplex vectors}\label{ss:v-def}

Let $n\ge 2$ be an integer, and consider a $2n$-simplex~$\Delta^{2n}$ whose vertices~$i$ we would like to denote by numbers from $i=1$ through $i=2n+1$. In this section, its boundary~$\partial \Delta^{2n}$ will serve us as the union of the two sides of our Pachner move $(n+1) \mapsto n$.

Let $X$ be any fixed set, called \emph{set of colors}.

\begin{definition}\label{d:col}
A \emph{coloring} of the $(2n-2)$-dimen\-sional faces~$v=\Delta^{2n-2} \subset \Delta^{2n}$, which we also call simply a coloring of~$\Delta^{2n}$, is any function assigning an element of~$X$, called \emph{color}, to each~$v$.
\end{definition}

All such colorings form the $n(2n+1)$-th Cartesian degree~$X^{\times n(2n+1)}$ of set~$X$.

\smallskip

Let
\begin{equation}\label{A}
\mathcal M = \begin{pmatrix} \alpha _1 & & \dots & & \alpha _{2n+1} \\ \beta _1 & & \dots & & \beta _{2n+1} \\ \gamma _1 & & \dots & & \gamma _{2n+1} \end{pmatrix}
\end{equation}
be a $3\times (2n+1)$ matrix whose entries are \emph{indeterminates} (algebraically independent elements) over a field~$F$. We do not specify this~$F$ at the moment, but the reader may think of $F=\mathbb Q$ ---the field of rational numbers, or $F=\mathbb F_q$ ---a finite field.

We now choose the color set to be the field
\begin{equation}\label{bF}
X = \mathcal F \stackrel{\mathrm{def}}{=} F(\alpha _1, \ldots, \gamma _{2n+1}) 
\end{equation}
of rational functions of the entries of matrix~$\mathcal M$ over~$F$.

We denote
\begin{equation}\label{dijk}
d_{ijk} = \left| \begin{matrix} \alpha _i & \alpha _j & \alpha _k \\ \beta _i & \beta _j & \beta _k \\ \gamma _i & \gamma _j & \gamma _k \end{matrix} \right|
\end{equation}
the determinant made of its $i$-th, $j$-th and $k$-th columns.

Recall (Subsection~\ref{ss:nota}) that a \emph{g-simplex} is any simplex of dimension~$(n-2)$.

\begin{definition}\label{d:he}
For a given g-simplex~$b=\Delta^{n-2}$, its corresponding \emph{g-simplex vector}~$e_b$ is the coloring of~$\Delta^{2n}$ whose component~$e_b|_v$ on each $v=\Delta^{2n-2}$ is the following product of determinants~\eqref{dijk} over the vertices of~$b$:
\begin{equation}\label{e_b}
e_b|_v = \prod_{i\in b} d_{ilm},
\end{equation}
where $l < m$ are the two vertices of~$\Delta^{2n}$ \emph{not} belonging to~$v$. Recall (Subsection~\ref{ss:nota}) that we can use, in such situation, notation $v=lm$ (and this regardless of whether $l<m$).
\end{definition}

Definition~\eqref{e_b} generalizes the \emph{edge vectors} introduced in~\cite{hepta_1} for the \emph{heptagon} relation, see~\cite[(16)]{hepta_1}. Note that
\begin{equation*}
e_b|_v = 0 \quad \text{ if } \quad b\not\subset v.
\end{equation*}

\begin{definition}\label{d:pc}
A \emph{g-coloring} of P-simplex~$\Delta^{2n}$ is any element of the linear subspace $V\subset F^{n(2n+1)}$ spanned by all g-simplex vectors. A g-coloring of any d-simplex~$\Delta^{2n-1} \subset \Delta^{2n}$ is the restriction of any g-coloring of~$\Delta^{2n}$ on this~$\Delta^{2n-1}$.
\end{definition}

\subsection{Linear relations between g-simplex vectors}\label{ss:v-rel}

\begin{proposition}\label{p:r}
Let $i_1\ldots i_{n-2}$ be an $(n-3)$-dimen\-sional face, and $j,k,l,m$ ---four vertices lying outside of it. Then, the following relation between four g-simplex vectors holds:
\begin{equation}\label{lr}
d_{klm}e_{i_1\ldots i_{n-2}j} - d_{jlm}e_{i_1\ldots i_{n-2}k} + d_{jkm}e_{i_1\ldots i_{n-2}l} - d_{jkl}e_{i_1\ldots i_{n-2}m} = 0.
\end{equation}
\end{proposition}

\begin{proof}
Consider the components of all summands in~\eqref{lr} for a face~$pq$. After substituting~\eqref{e_b} and factoring out $\prod_{i\in \{i_1,\ldots, i_{n-2}\}}e_{ipq}$, what remains is just a Pl\"ucker bilinear relation.
\end{proof}

There are $n+2$ vertices involved in linear relation~\eqref{lr}: $i_1,\ldots, i_{n-2}$ plus $j,k,l,m$. The following proposition shows that there are \emph{no} linear relations in which not more than $n+1$ vertices are involved, or, in other words, no linear relations within any simplex~$\Delta^n$.

\begin{proposition}\label{p:n}
Vectors~$e_b$ for all $b\subset \Delta^n$, where $\Delta^n$ is any chosen $n$-simplex, are linearly independent.
\end{proposition}

\begin{proof}
Let $\Delta^n=i_1\ldots i_{n+1}$ be an $n$-simplex, and let $i_k\ne i_l$ be two of its vertices. Suppose a linear relation
\begin{equation}\label{bl}
\sum _{b\subset \Delta^n} \lambda _b e_b = 0
\end{equation}
holds. Consider the component of~\eqref{bl} corresponding to face~$i_k i_l$. The only g-simplex~$b$ that is contained in this face is
\begin{equation*}
b_{kl} = i_1 \ldots \hat i_k \ldots \hat i_l \ldots i_{n+1},
\end{equation*}
where ``hat'' means omission (of $i_k$ and~$i_l$). We get thus
\begin{equation*}
\lambda _{b_{kl}} e_{b_{kl}} |_{kl} = 0 \;\; \Rightarrow \;\; \lambda _{b_{kl}} = 0.
\end{equation*}
As we can take arbitrary non-coinciding $i_k$ and~$i_l$, all the lambdas in~\eqref{bl} are zero.
\end{proof}

One~$\Delta^n$ has $\binom{n+1}{n-1}=\frac{n(n+1)}{2}$ faces of dimension~$n-2$, so it generates a $\frac{n(n+1)}{2}$-dimen\-sional linear space of g-colorings for~$\Delta^{2n}$.

On the other hand, it is an easy exercise to show, using linear relations~\eqref{lr}, that~$e_b$ for any~$b\subset \Delta^{2n}$ can be expressed through the vectors~$e_b$ taken for $b\subset \Delta^n$ with one chosen~$\Delta^n$. Hence, the following proposition holds.

\begin{proposition}\label{p:d}
The dimension of the space~$V_g(\Delta^{2n})$ of g-colorings of~$\Delta^{2n}$ is
\begin{equation}\label{d}
\dim V_g(\Delta^{2n}) = \frac{n(n+1)}{2}.
\end{equation}
\qed
\end{proposition}

\subsection[Restriction on one $(2n-1)$-simplex]{Restriction on one $\boldsymbol{(2n-1)}$-simplex}\label{ss:2n}

In the previous subsection, we considered colorings of the whole $2n$-simplex, which we also call P-simplex (see Subsection~\ref{ss:nota}), and whose boundary is the union of the l.h.s.\ and r.h.s.\ of a Pachner move. Now, we are going to consider colorings of just one $(2n-1)$-simplex~$\Delta^{2n-1}$, which we also call d-simplex.

Take $\Delta^{2n-1}=m$ ---remember (see again Subsection~\ref{ss:nota}) that in our ``complemental'' notations this means the simplex \emph{without} vertex~$m$. We denote $e_b|_m$ the restriction of~$e_b$ onto the faces of~$m$. There is the following analogue of Proposition~\ref{p:r}.
\begin{proposition}\label{p:rr}
There are the following three-term linear relations:
\begin{equation}\label{lr_2n-1}
d_{klm}e_{i_1\ldots i_{n-2}j}|_m - d_{jlm}e_{i_1\ldots i_{n-2}k}|_m + d_{jkm}e_{i_1\ldots i_{n-2}l}|_m = 0.
\end{equation}
\end{proposition}

\begin{proof}
This follows at once from~\eqref{lr}; the fourth term disappears because $e_{i_1\ldots i_{n-2}m}$ has only zero components on~$m$.
\end{proof}

\begin{proposition}\label{p:nr}
Vectors~$e_b|_m$ for all $b\subset \Delta^{n-1}$, where $\Delta^{n-1}$ is any chosen $(n-1)$-simplex not containing vertex~$m$, are linearly independent.
\end{proposition}

\begin{proof}
In analogy with the proof of Proposition~\ref{p:n}, consider a relation
\begin{equation}\label{blr}
\sum _{b\subset \Delta^{n-1}} \lambda _b e_b|_m = 0.
\end{equation}
Consider the component of~\eqref{blr} corresponding to face~$im$, where $i\in \Delta^{n-1}$. The only g-simplex~$b$ that is contained in this face is that with all vertices of~$\Delta^{n-1}$ except~$i$, denote this g-simplex~$b_i$. We get thus
\begin{equation*}
\lambda _{b_i} e_{b_i} |_{im} = 0 \;\; \Rightarrow \;\; \lambda _{b_i} = 0.
\end{equation*}
As we can take arbitrary~$i$, all the lambdas in~\eqref{blr} are zero.
\end{proof}

One~$\Delta^{n-1}$ has $\binom{n}{n-1}=n$ faces of dimension~$n-2$, so it generates an $n$-dimen\-sional linear space of g-colorings for~$\Delta^{2n-1}$.

On the other hand, it is an easy exercise to show, using linear relations~\eqref{lr_2n-1}, that $e_b|_m$ for any $b \subset \Delta^{2n-1}\not \ni m$ can be expressed through the vectors~$e_b|_m$ taken for $b \subset \Delta^{n-1}$ with one chosen~$\Delta^{n-1}$ not containing~$m$. Hence, the following proposition holds.

\begin{proposition}\label{p:d-one}
The dimension of the space of g-colorings of~$\Delta^{2n-1}$ is~$n$.
\qed
\end{proposition}

\subsection{Dependence between the face colors of a d-simplex in matrix form}\label{ss:mat}

Figure~\ref{fig:A}
\begin{figure}
 \begin{center}
  \includegraphics[scale=0.75]{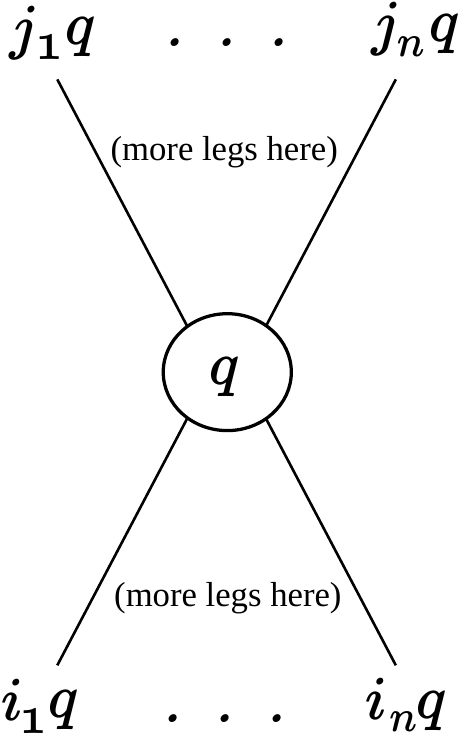}
 \end{center}
 \caption{Matrix $A^{(q)}$. Legs corresponding to the input faces are at the \emph{bottom}; legs corresponding to the output faces are at the \emph{top}; the matrix acts on \emph{rows}, that is, from the \emph{right}}
 \label{fig:A}
\end{figure}
depicts d-simplex~$i_1\ldots i_n j_1\ldots j_n$ in a `dual' way: the vertex---or actually the circle with the letter~$q$---symbolizes the simplex itself, while the edges, or `legs', correspond to its faces. We use `complemental' notations (recall Subsection~\ref{ss:nota}): numbers $i_1,\ldots, i_n, j_1,\ldots, j_n,q$ are supposed to make a permutation of numbers $1,\ldots,2n+1$, so our simplex is the $(2n-1)$-simplex \emph{without} vertex~$q$; similarly, each leg corresponds to the $(2n-2)$-face not having the two vertices with which the leg is labeled.

Proposition~\ref{p:d-one} suggests that we can take arbitrary colors for $n$ of~$2n$ faces of our simplex~$q$, let these be $i_1q,\ldots, i_nq$, then the colors of $j_1q,\ldots, j_nq$ are expected to be determined uniquely. The following reasoning shows that it is indeed so.

Consider $(n-2)$-simplex vector~$e_{i_2\ldots i_n}$. As $(2n-2)$-faces denoted $i_2 q$, \ldots, $i_n q$ do \emph{not} contain $i_2$, \ldots, $i_n$, respectively, the corresponding components of~$e_{i_2\ldots i_n}$ vanish. Other components are given by formula~\eqref{e_b}, so, the components of~$e_{i_2\ldots i_n}$ on the $n$ lower legs of Figure~\ref{fig:A}---call them together, for a moment, $e_{i_2\ldots i_n}|_{\mathrm{lower}}$---are
\begin{equation}\label{eb-niz}
e_{i_2\ldots i_n}|_{\mathrm{lower}} = \begin{pmatrix} d_{i_2i_1q}\cdots d_{i_ni_1q} & 0 & \dots & 0 \end{pmatrix},
\end{equation}
while the components on the $n$ upper legs are
\begin{equation}\label{eb-verh}
e_{i_2\ldots i_n}|_{\mathrm{upper}} = \begin{pmatrix} d_{i_2j_1q}\cdots d_{i_nj_1q} && \dots && d_{i_2j_nq}\cdots d_{i_nj_nq} \end{pmatrix}.
\end{equation}

Remembering that
\begin{equation}\label{A_row_1}
e_{i_2\ldots i_n}|_{\mathrm{lower}}\, A^{(q)} = e_{i_2\ldots i_n}|_{\mathrm{upper}},
\end{equation}
we get immediately the first row of matrix~$A^{(q)}$. Similarly, we get other rows, and the final answer for~$A^{(q)}$ is
\begin{equation}\label{A-expl}
\left( A^{(q)} \right)_{iq}^{jq} = \prod_{\substack{i'\ne i\\ i'q\;\mathrm{inputs}}} \frac{d_{i'jq}}{d_{i'iq}}\,. 
\end{equation}
Proposition~\ref{p:d-one} guarantees that there will arise no contradiction if we use other g-colorings for determining~$A^{(q)}$.

In such context, we call legs $i_1q,\ldots, i_nq$ \emph{input}, and legs $j_1q,\ldots, j_nq$ \emph{output}. In our pictures (as well as in~\cite{DK}), input legs lie below the corresponding vertex, and output legs---above. It is clear from our reasoning that \emph{any} $n$ legs---that is, faces of a d-simplex---can be declared input; the corresponding matrix will be given by~\eqref{A-expl} with the relevant permutation of indices.

\subsection[Odd polygon relation for matrices~$A^{(q)}$]{Odd polygon relation for matrices~$\boldsymbol{A^{(q)}}$}\label{ss:h}

Our odd polygon relation is the equality of two products of $\frac{n(n+1)}{2}\times \frac{n(n+1)}{2}$ matrices called~$A_{\mathcal B_q}^{(q)}$, each being a direct sum of matrix~\eqref{A-expl} and an identity matrix. More specifically: think of each~$A_{\mathcal B_q}^{(q)}$ as acting on $\frac{n(n+1)}{2}$-\emph{rows}, then by definition it acts nontrivially only on the entries with numbers in the set~$\mathcal B_q$ (which we are going to specify), while all the other entries remain intact under~$A_{\mathcal B_q}^{(q)}$. This implies of course that the cardinality of each set~$\mathcal B_q$ equals~$n$.

The definition of~$\mathcal B_q$ is as follows.

\begin{definition}\label{d:B}
Write all pairs~$(i,j)$ of \emph{odd} numbers from~1 through~$2n+1$, such that $i<j$, in the lexicographic order:
\begin{equation}\label{13}
\begin{aligned}
(1,3),\; (1,5),\; \ldots,\; (1,2n+1),\;\ldots,\; & \\
\boldsymbol{(2k-1,2k+1),\;(2k-1,2k+3),\;} & \boldsymbol{\dots, \; (2k-1,2n+1),}\; \\
& \ldots,\; (2n-1,2n+1),
\end{aligned}
\end{equation}
where we highlighted in bold a typical subsequence. There are $\frac{n(n+1)}{2}$ members in sequence~\eqref{13}, and by definition, $\mathcal B_q$ for an odd~$q$ consists of the positions of such pairs that include~$q$.

Now write all pairs~$(i,j)$ of \emph{even} numbers from~2 through~$2n$, such that $i\le j$ (pay attention to the non-strict inequality!), in the lexicographic order:
\begin{equation}\label{22}
\begin{aligned}
& (2,2),\; (2,4),\; \ldots,\; (2,2n),\;\ldots,\; \\
& \boldsymbol{(2k,2k),\;(2k,2k+2),\;\dots, \; (2k,2n),}\; \ldots,\; (2n,2n),
\end{aligned}
\end{equation}
where we also highlighted in bold a typical subsequence. There are again $\frac{n(n+1)}{2}$ members in sequence~\eqref{22}, and by definition, $\mathcal B_q$ for an even~$q$ consists again of the positions of such pairs that include~$q$.
\end{definition}

It is an easy exercise to check that Definition~\ref{d:B} gives the same as the definition of~$\mathcal B_q$ given in~\cite[Subsection~II.B]{DK}.

The following proposition states a fundamental property of sets~$\mathcal B_q$.

\begin{proposition}\label{p:u}
The intersection $\mathcal B_q \cap \mathcal B_r$ for $q\ne r$ consists of exactly one element.
\end{proposition}

\begin{proof}
Suppose that, for instance, $q<r$. If $q$ and~$r$ are either both odd or both even, then $\mathcal B_q \cap \mathcal B_r$ obviously consists exactly of the position of pair $(q,r)$.

To consider other cases, note that the bijection between the odd and even pairs conserving their order is given by $(i,j)\leftrightarrow (i+1,j-1)$. So, if $q$ is odd and $r$ is even (and still $q<r$), then $\mathcal B_q \cap \mathcal B_r$ consists, again quite obviously, of the position of pair $(q,r+1)$, or, which is the same, of pair $(q+1,r)$. And if $q$ is even and $r$ is odd, the relevant pairs with the same position are $(q,r-1)$ and $(q-1,r)$.
\end{proof}

\begin{definition}\label{d:p}
We call the following equality \emph{odd polygon (odd-gon) relation}:
\begin{equation}\label{p}
A_{\mathcal B_1}^{(1)} A_{\mathcal B_3}^{(3)} \cdots A_{\mathcal B_{2n+1}}^{(2n+1)} = A_{\mathcal B_{2n}}^{(2n)} A_{\mathcal B_{2n-2}}^{(2n-2)} \cdots A_{\mathcal B_2}^{(2)} .
\end{equation}
\end{definition}

Before proving~\eqref{p}, we study its combinatorial structure. We think of both sides of~\eqref{p} as acting on $\frac{n(n+1)}{2}$-rows. Let there be an arbitrary $\frac{n(n+1)}{2}$-row, call it \emph{initial} row. The action of each~$A_{\mathcal B_q}^{(q)}$ is as follows: take the $\frac{n(n+1)}{2}$-row resulting after the action of all~$A_{\mathcal B_i}^{(i)}$ staying to the left of~$A_{\mathcal B_q}^{(q)}$ on the initial row, then take its $n$-subrow with entry numbers (positions) in~$\mathcal B_q$---we call these \emph{inputs} for~$A^{(q)}$---and act with~$A^{(q)}$ upon it; we call the entries of the resulting $n$-row \emph{outputs} for~$A^{(q)}$. Entries with positions outside~$\mathcal B_q$ stay intact; they remain thus either initial or outputs of a previous~$A^{(i)}$.

We call the results of action of the l.h.s.\ and r.h.s.\ of~\eqref{p} on the initial row \emph{final} rows; we will show soon, in Theorem~\ref{th:p}, that these two rows coincide.

\begin{definition}\label{d:l}
We say that matrix~$A^{(q)}$ is \emph{linked} to~$A^{(r)}$ if 
\begin{itemize}\itemsep 0pt
 \item either they are both in the same side of~\eqref{p}, and an output of one of these matrices serves as an input for the other,
 \item or they are in the different sides of~\eqref{p}, and have among their inputs the same entry of the initial row,
 \item or they are in the different sides of~\eqref{p}, and have among their outputs entries with the same position in the final row (these entries will actually coincide, but we have not proved it as yet!).
\end{itemize}
\end{definition}

\begin{proposition}\label{p:l}
Any $A^{(q)}$ is linked to \emph{all other}~$A^{(r)}$.
\end{proposition}

\begin{proof}
Any $A^{(q)}$ has $n$ inputs and $n$ outputs; each of these links $A^{(q)}$ to some~$A^{(r)}$, and we have to prove that these~$A^{(r)}$ are all different. If two links (inputs, outputs, or mixed) leading to some $A^{(r)}$ and~$A^{(s)}$ are at different positions, then $A^{(r)}$ and~$A^{(s)}$ cannot coincide because of Proposition~\ref{p:u}.

Let now two links be the input and output at the same position~$i$. If they could lead to the same~$A^{(r)}$, this would mean that $A^{(r)}$ is in the different side of~\eqref{p}, and that $i$ belongs to $\mathcal B_q$ and~$\mathcal B_r$, but to none other of sets~$\mathcal B_k$. But one can see from Definition~\ref{d:B} that any position belongs to either three or four sets~$\mathcal B_k$.
\end{proof}

Proposition~\ref{p:l} shows that the structure of~\eqref{p} corresponds to $\partial \Delta^{2n}$ divided in two parts. To make this statement more visual, it makes sense to think of~\eqref{p} in terms of its \emph{diagram}.

\begin{definition}\label{d:d}
The \emph{diagram} of (both sides of)~\eqref{p} or a similar relation (like~\eqref{gog} below) is made of vertices like that in Figure~\ref{fig:A}; if an output of one matrix is an input of another, the corresponding legs are joined together.
\end{definition}

A leg linking $A^{(q)}$ and~$A^{(r)}$ (in the sense of Definition~\ref{d:l}) corresponds to $(2n-2)$-face~$qr$. This face is \emph{boundary} provided $A_{\mathcal B_q}^{(q)}$ and~$A_{\mathcal B_r}^{(r)}$ are in the different sides of~\eqref{p}.

\begin{example}\label{x:h}
The \emph{heptagon} relation, corresponding to $n=3$, reads
\begin{equation}\label{hepta}
A_{123}^{(1)}A_{145}^{(3)}A_{246}^{(5)}A_{356}^{(7)}=A_{356}^{(6)}A_{245}^{(4)}A_{123}^{(2)}\,.
\end{equation}
Its diagram is shown in Figure~\ref{fig:hepta-new}.
\begin{figure}
 \begin{center}
  \includegraphics[scale=1.25]{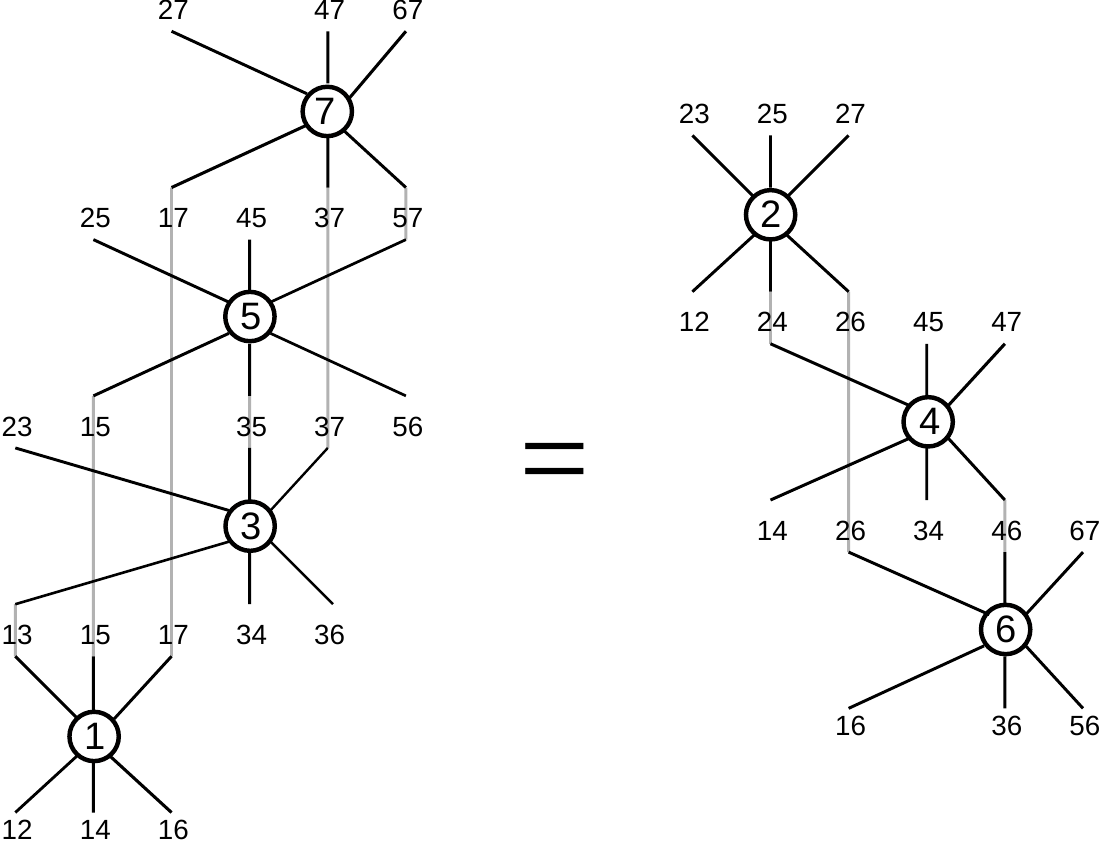}
 \end{center}
 \caption{Heptagon relation diagram}
 \label{fig:hepta-new}
\end{figure}
The corresponding Pachner move replaces the cluster of 5-simplices 1, 3, 5 and 7 with the cluster of 5-simplices 2, 4 and~6.
\end{example}

\begin{proposition}\label{p:e}
Take any vector~$e_b$ and put its components on the corresponding legs of the polygon relation diagram. Then, they will agree with all matrices~$A^{(q)}$: the action of each~$A^{(q)}$ on the components on its incoming legs gives exactly the components on its outgoing legs.
\end{proposition}

\begin{proof}
This follows immediately from the fact that the structure of the odd polygon relation~\eqref{p} corresponds exactly to the Pachner move $n+1 \to n$.
\end{proof}

\begin{theorem}\label{th:p}
Odd polygon relation~\eqref{p} does hold for matrices~\eqref{A-expl}.
\end{theorem}

\begin{proof}
This follows from Proposition~\ref{p:e} and the fact that the dimension of the linear space of g-colorings is $n(n+1)/2$, see~\eqref{d}, which is the same as the number of input legs of the l.h.s.\ or r.h.s.\ of the diagram of~\eqref{p}.
\end{proof}
Relation~\eqref{p} guarantees that the boundary colorings mentioned in Definition~\ref{d:F} are the same for one specific Pachner move, namely, $(n+1)\mapsto n$ and such the there are d-simplices $1,3,\ldots,2n+1$ in its initial cluster, and $2,4,\ldots,2n$ in its final cluster; d-simplex~$q$ means here, as we remember, the simplex with all vertices $1,\ldots,2n+1$ except~$q$. If there are other d-simplices in both sides of a move $(n+1)\mapsto n$, there is still no big problem to get a corresponding odd-gon rela

\section{Full polygon---relations corresponding to all Pachner moves}\label{s:F}

\subsection{Gauge transformations of the colors}\label{ss:gt}

First, a small preparatory work.

Odd-gon relation~\eqref{p} remains valid if we renormalize colors~$x_u$ of faces~$u$, that is, make a substitution
\begin{equation}\label{g}
x_u = \lambda _u x_u^{\mathrm{new}},
\end{equation}
where $\lambda _u$ are nonzero elements of the field~$\mathcal F$.

For a matrix~$A^{(q)}$ with input faces $u_1,\ldots,u_n$ and output faces $v_1,\ldots,v_n$, this means the following transformation with two diagonal matrices:
\begin{equation}\label{lg}
A^{(q)} \mapsto \left(A^{(q)}\right)^{\mathrm{new}} = \begin{pmatrix} \lambda _{u_1} & & \\ & \ddots & & \\ & & \lambda _{u_n} \end{pmatrix} A^{(q)} \begin{pmatrix} \lambda _{v_1} & & \\ & \ddots & & \\ & & \lambda _{v_n} \end{pmatrix}^{-1}
\end{equation}
(remember that our matrices act on the \emph{rows} and from the right!)

\begin{definition}\label{d:g}
The transformation~\eqref{g} of colors, as well as the corresponding transformation~\eqref{lg} of matrices, are called \emph{gauge} transformations.
\end{definition}

\subsection{Full polygon}\label{ss:fp}

\begin{definition}\label{d:pi}
Suppose that, for each d-simplex~$v$, a subset~$R_v$ of the set of all its colorings is given called \emph{permitted colorings}. Then, a coloring of a cluster of d-simplices is called \emph{permitted} if its restrictions on all d-simplices are permitted.
\end{definition}

\begin{definition}\label{d:F}
We say that a rule defining permitted colorings for d-simplices satisfies \emph{full polygon} in dimension~$\nu$ if the permitted colorings of the boundary are the same for the initial and final clusters of any $\nu$-dimen\-sional Pachner move~\eqref{kP}.
\end{definition}

In this subsection, we define permitted colorings of a d-simplex to be the same as its g-colorings (Definition~\ref{d:pc}). We denote the vertices of simplices involved in any Pachner move by numbers $1,\ldots,2n+1$.
tion: we just make a relevant permutation of numbers $1,\ldots,2n+1$. One small thing to be taken into account is that, with this permutation, colors of some faces may change their signs with respect to definition~\eqref{e_b}, because the condition $l<m$ (just below~\eqref{e_b}) may change to its reverse. The corresponding change of matrices~$A^{(q)}$ with respect to~\eqref{A-expl} is, however, just a gauge transformation, with some lambdas in~\eqref{lg} equal to unity and some to minus unity.

For a Pachner move $(n+1+k)\mapsto (n-k)$, with $k>0$, move $k$ matrices~$A^{(q)}$ from the r.h.s.\ of~\eqref{p} to the l.h.s., by multiplying both sides by their inverses $B^{(q)}=\left(A^{(q)}\right)^{-1}$, for instance, from the right:
\begin{equation}\label{gog}
A_{\mathcal B_1}^{(1)} A_{\mathcal B_3}^{(3)} \cdots A_{\mathcal B_{2n+1}}^{(2n+1)}
\cdot B_{\mathcal B_2}^{(2)} B_{\mathcal B_4}^{(4)} \cdots B_{\mathcal B_{2k}}^{(2k)} 
= A_{\mathcal B_{2n}}^{(2n)} A_{\mathcal B_{2n-2}}^{(2n-2)} \cdots A_{\mathcal B_{2k+2}}^{(2k+2)} ,
\end{equation}
then make a relevant permutation of vertices, and the gauge transformation if necessary, as we just did for move $(n+1) \mapsto n$.

Any matrix $A_{\mathcal B_q}^{(q)}$ or~$B_{\mathcal B_q}^{(q)}$ in~\eqref{gog} is still linked to any other matrix in the sense of Definition~\ref{d:l}, so its structure corresponds indeed to the move $(n+1+k)\mapsto (n-k)$. Note also that matrices~$B_{\mathcal B_q}^{(q)}$ are in fact of the very same form~\eqref{A-expl} as~$A_{\mathcal B_q}^{(q)}$, with only input and output legs interchanged, exactly as the structure of~\eqref{gog} requires.

Finally, for a Pachner move having \emph{less} d-simplices in the initial cluster than in the final cluster, just take the relation for its inverse, and interchange the l.h.s.\ and r.h.s.

This gives us the following theorem.

\begin{theorem}\label{th:F}
The full odd-gon does hold if the permitted colorings for d-simplices are as in Definition~\ref{d:pc}, or, equivalently, they are defined using matrices~\eqref{A-expl}.
\qed
\end{theorem}

Note that the r.h.s.\ of~\eqref{gog} may have less inputs---and outputs---than the l.h.s.\ (and it actually always happens when $k>0$). The `superfluous' output values in the l.h.s.\ coincide with the corresponding input values (because they obviously coincide in the r.h.s). Topologically, any `superfluous' input and its corresponding output belong to the \emph{same inner face} of the initial cluster, while the inputs and outputs of the r.h.s.\ form the \emph{boundary} of either initial or final cluster.

As a permitted coloring of either l.h.s.\ or r.h.s.\ of~\eqref{gog} is determined by its (arbitrary) inputs, and the l.h.s.\ has, generally, more inputs than the r.h.s., it follows that to a permitted coloring of r.h.s.\ corresponds a whole linear space of permitted colorings of l.h.s., of dimension equal to the number of `superfluous' inputs.

\section{Polygon cohomology}\label{s:coho}

\subsection{Nonconstant polygon cohomology: generalities}\label{ss:cg}

In this subsection, we define ``nonconstant polygon cohomology'' in a general context. It works equally well for both even and odd polygons, so, in this section, we speak of a ``$(\nu+2)$-gon'', where meaningful values of~$\nu$ can be $\nu=3,4,\ldots$, although, as we already stated~\eqref{nun}, we deal mostly with the odd $\nu=2n-1$ in this paper. Our $(\nu+2)$-gon corresponds to a Pachner move in a $\nu$-dimen\-sional PL manifold.

Our definition will depend on a chosen simplicial complex~$K$. In principle, $K$ can be of any dimension, although the main work in this paper will take place in the standard $(\nu+1)$-simplex~$K=\Delta^{\nu+1}$, whose boundary is the union of the l.h.s.\ and r.h.s.\ of any $\nu$-dimen\-sional Pachner move.

Suppose that every $(\nu-1)$-face~$u\subset K$ is colored by some element $\mathsf x_u\in X$ of a set~$X$ of colors (for instance, field $X=\mathcal F$ considered in the preceding sections), and that a subset~$R_v$ of permitted colorings is defined in the set of all colorings of every $\nu$-simplex~$v$.

\begin{example}\label{x:pe}
For the P-simplex~$\Delta^{\nu+1} = \Delta^{2n}$ of Section~\ref{s:r}, permitted colorings of its $\nu$-dimen\-sional faces are the restrictions of the g-coloring of Definition~\ref{d:pc} on these faces.
\end{example}

We define also the set of `permitted' colorings for any simplex $\Delta\subset K$ of any dimension.

\begin{definition}\label{d:pD}
A coloring of a simplex~$\Delta\subset K$ is \emph{permitted} provided its restrictions on all $\nu$-faces of~$\Delta$ are permitted.
\end{definition}

This agrees, of course, with our earlier Definition~\ref{d:pi}.

\begin{proposition}\label{p:pa}
Permitted colorings for P-simplex~$\Delta^{2n}$, in the situation of Section~\ref{s:r}, coincide with its g-colorings (Definition~\ref{d:pc}).
\end{proposition}

\begin{proof}
Clearly, any g-coloring is also permitted. On the other hand, the dimension of the space of permitted colorings is the number of inputs in either side of~\eqref{p}, which is $\frac{n(n+1)}{2}$ and coincides with the dimension of the space of g-colorings (Proposition~\ref{p:d}).
\end{proof}

For other complexes, there may be more permitted colorings than g-colorings, see Examples \ref{x:V3} and~\ref{x:V5} below.

It is implied in Definition~\ref{d:pD} that \emph{all} colorings $\mathsf x\in X$ are permitted for an individual $(\nu-1)$-face. Also, Definition~\ref{d:pD} is void of course for $\dim\Delta < \nu-1$; it has a nontrivial meaning for~$\dim\Delta > \nu$.

The set of all permitted colorings of an $m$-simplex $i_0\dots i_m$ will be denoted~$\mathfrak C_{i_0\dots i_m}$. We assume here that the vertices of any simplex are ordered: $i_0<\ldots <i_m$.

\begin{definition}\label{d:nc}
An \emph{$m$-cochain}~$\mathfrak c$ taking values in abelian group~$G$, for $m\ge \nu-1$, consists of arbitrary mappings
\begin{equation}\label{nc}
\mathfrak c_{i_0\dots i_m}\colon\;\,\mathfrak C_{i_0\dots i_m} \to G
\end{equation}
for \emph{all} $m$-simplices $\Delta^m=i_0\dots i_m\subset K$. The \emph{coboundary}~$\delta \mathfrak c$ of~$\mathfrak c$ consists of mappings $(\delta \mathfrak c)_{i_0\dots i_{m+1}}$ acting on a permitted coloring~$r\in \mathfrak C_{i_0\dots i_{m+1}}$ of $(m+1)$-simplex $i_0\dots i_{m+1}$ according to the following formula:
\begin{equation}\label{cb}
(\delta \mathfrak c)_{i_0\dots i_{m+1}} (r) = \sum _{k=0}^{m+1} (-1)^k\, \mathfrak c_{i_0\dots \widehat{i_k} \dots i_{m+1}} (r|_{i_0\dots \widehat{i_k} \dots i_{m+1}}),
\end{equation}
where each $r|_{i_0\dots \widehat{i_k} \dots i_{m+1}}$---the restriction of~$r$ onto the $m$-simplex $i_0\dots \widehat{i_k} \dots i_{m+1}$---is of course a permitted coloring of this latter simplex.
\end{definition}

\begin{definition}\label{d:ncc}
Nonconstant polygon cohomology is the cohomology of the following \emph{heptagon cochain complex}:
\begin{equation}\label{hcc}
0 \to C^{\nu-1} \stackrel{\delta}{\to} C^{\nu} \stackrel{\delta}{\to} C^{\nu+1} \stackrel{\delta}{\to} \dots\, ,
\end{equation}
where $C^m$ means the group of all $m$-cochains.
\end{definition}

\subsection{Special kinds of cohomology: polynomial, symmetric bilinear, bipolynomial}\label{ss:q}

It turns out that there are some interesting variations of the cochain definition~\eqref{nc}. For instance, preprint~\cite{cubic} (although devoted to \emph{constant hexagon} cohomology) suggests that \emph{polynomials} may be used instead of general functions~\eqref{nc}---of course, in a situation where the notion of polynomial in the variables determining a permitted coloring makes sense.

\begin{definition}\label{d:q}
We call a complex~\eqref{hcc} \emph{polynomial} if $G=F$ is the same field~$F$ that we are using for the set $X=F$ of colors, and mappings~$\mathfrak c_{i_0\dots i_m}$ in~\eqref{nc} are polynomials (i.e., depend polynomially on coordinates) on the linear spaces of permitted colorings of corresponding $m$-simplices~$i_0\ldots i_m$.
\end{definition}

We can also ``double'' the colorings and permitted colorings, introducing the Cartesian square $X\times X$ as a new color set, and Cartesian squares $R_v\times R_v$ as new sets of permitted colorings for each $\nu$-simplex~$v$. This construction will be of special interest for us in the \emph{symmetric bilinear} and \emph{bipolynomial} cases.

\begin{definition}\label{d:b}
A \emph{symmetric bilinear cochain} consists of symmetric bilinear functions
\begin{equation}\label{bc}
\mathfrak c_{i_0\dots i_m}\colon\quad \mathfrak C_{i_0\dots i_m} \times \mathfrak C_{i_0\dots i_m} \to F,
\end{equation}
of a \emph{pair}~$(r_1,r_2)$ of permitted colorings. The \emph{codifferential} for such cochains is obtained by changing $r$ to $(r_1,r_2)$ in~\eqref{cb}:
\begin{equation}\label{cb2}
(\delta \mathfrak c)_{i_0\dots i_{m+1}} (r_1,r_2) = \sum _{k=0}^{m+1} (-1)^k\, \mathfrak c_{i_0\dots \widehat{i_k} \dots i_{m+1}} (r_1|_{i_0\dots \widehat{i_k} \dots i_{m+1}},\, r_2|_{i_0\dots \widehat{i_k} \dots i_{m+1}}).
\end{equation}
\end{definition}

\begin{definition}\label{d:bp}
A \emph{bipolynomial cochain} consists of functions~\eqref{bc} that are polynomial in both $r_1$ and~$r_2$. The codifferential is of course given again by formula~\eqref{cb2}.
\end{definition}

In a characteristic $\ne 2$, symmetric bilinear cochains are of course just \emph{polarizations} of quadratic homogeneous polynomial cochains. Preprint~\cite{cubic} suggests, however, that characteristic~$2$ is interesting and deserves attention.

In any characteristic, symmetric bilinear mapping~\eqref{bc} can be treated as a \emph{scalar product} of two permitted colorings of simplex~$i_0\dots i_m$; to define a symmetric bilinear cochain is the same as to define scalar products for all corresponding simplices. This language of scalar products will be used extensively in what follows.

\section[Constructing invariant of a PL manifold from a polynomial $(2n-1)$-cocycle]{Constructing invariant of a PL manifold from a polynomial $\boldsymbol{(2n-1)}$-cocycle}\label{s:f}

\subsection{Dual parameters}\label{ss:dp}

In Section~\ref{s:r}, we parameterized g-colorings for a Pachner move by entries of matrix~$\mathcal M$~\eqref{A}. In order to consider more general PL manifolds, it makes sense to switch to a different---`dual'---parameterization. As a first step, we introduce here these dual parameters for that same P-simplex whose boundary is the union of the two sides of the Pachner move.

The three rows of matrix~$\mathcal M$ span a three-dimen\-sional linear subspace~$L$ in the $(2n+1)$-dimen\-sional space~$\mathcal L$ of $(2n+1)$-rows (over the field~$\mathcal F$~\eqref{bF}). Its \emph{orthogonal},~$L^{\perp}$, is $(2n-2)$-dimen\-sional, lies in the dual space~$\mathcal L^*$, and can be represented as spanned by the $(2n-2)$ \emph{columns} of a matrix
\begin{equation}\label{N}
\mathcal N = \begin{pmatrix} \mu _{1,1} && \dots && \mu _{1,2n-2} \\
                             \vdots && \cdots && \vdots\\
                             \mu _{2n+1,1} && \dots && \mu _{2n+1,2n-2} \end{pmatrix}
\end{equation}
Denote $p_{i_1\dots i_{2n-2}}$ the determinant made of \emph{rows} $i_1,\dots, i_{2n-2}$ of~$\mathcal N$. A well-known fact about the duality between such determinants and the determinants $d_{jkl}$~\eqref{dijk} is that~$\mathcal N$ can be normalized so that the following relation will hold:
\begin{equation}\label{pd}
p_{i_1\dots i_{2n-2}} = (-1)^{i_1+\dots +i_{2n-2}} d_{jkl},
\end{equation}
where
\begin{equation}\label{pdc}
 \begin{aligned}
  i_1<\dots <i_{2n-2},\qquad j<k<l, \\[.5ex]
  \{i_1\dots i_{2n-2}\} \cup \{j,k,l\} = \{1,\ldots, 2n+1\}.
 \end{aligned}
\end{equation}

\subsection{g-colorings in terms of dual parameters}\label{ss:gd}

We are going to give an alternative definition of $(n-2)$-simplex vectors, and formulate it in such way that it will work not only for a Pachner move (taking place in the boundary of a P-simplex), but also for more general simplicial complexes. We introduce notation~$\mathcal N_i$ which will mean, in this subsection, simply the $i$-th row of matrix~$\mathcal N$~\eqref{N} ---because here we are still in the situation of one Pachner move or, in other words, P-simplex~$\Delta^{2n}$ with its $(2n+1)$ vertices---but $\mathcal N_i$ will be a more general row in the following sections.

\begin{definition}\label{d:f}
For a face~$v$ and g-simplex $b\subset v$, we set
\begin{equation}\label{ebd}
e_b|_v = \prod_{i\in b} (-1)^{\pos\nolimits_v i} p_{v,\hat i} \, ,
\end{equation}
where $\pos\nolimits_v i$ is the position of vertex~$i\in b$ among the vertices of face~$v$ taken in the increasing order, and $p_{v,\hat i}$ means the determinant made of all rows~$\mathcal N_j$ corresponding to vertices of~$v$ except~$j=i$, taken in the increasing order of~$j$.
\end{definition}

\begin{proposition}
Definitions \ref{d:f} and~\ref{d:he} give, for~$\Delta^{2n}$, the same~$e_b|_v$.
\end{proposition}

\begin{proof}
This follows from \eqref{e_b}, \eqref{ebd}, \eqref{pd}, and the fact that
\begin{equation*}
i_k = \pos\nolimits_v i_k + H(i_k-l) + H(i_k-m), 
\end{equation*}
where
\begin{equation*}
H(x)=\begin{cases}1,& x>0 \\ 0,& x\le 0 \end{cases}
\end{equation*}
is the Heaviside step function---that is, $H(i_k-l) + H(i_k-m)$ shows how many of two numbers $l$ and~$m$ precede~$i_k$, and this is exactly how many times the sign is changed in the multiplier~$d_{i_k lm}$ in~\eqref{e_b} compared with the case where $i_k$, $l$ and~$m$ go in the increasing order.
\end{proof}

\subsection{Definition of a manifold invariant as a `stable polynomial' up to linear transforms of its variables}\label{ss:i}

Let $M$ be a $(2n-1)$-dimen\-sional closed \emph{triangulated} PL manifold, and $F$ a prime field~$\mathbb F_p$ or the field~$\mathbb Q$ of rational numbers. If $\characteristic F\ne 2$, let also $M$ be \emph{oriented}.

We put in correspondence to each vertex~$i$ in the triangulation of~$M$ the following $(2n-2)$-row of parameters---indeterminates over~$F$ (analogue to a row of matrix~$\mathcal N$~\eqref{N}):
\begin{equation}\label{Ni}
\mathcal N_i = \begin{pmatrix} \mu _{i,1} && \dots && \mu _{i,2n-2} \end{pmatrix} .
\end{equation}
Additionally, if a new vertex arises as a result of a Pachner move made on the triangulation of~$M$, we put a row of the same kind~\eqref{Ni} in correspondence to it.

As usual, we think of all the vertices involved in our reasonings as numbered by natural numbers, which implies that they are ordered and gives sense to inequalities like $i_1<i_2$ for vertices $i_1$ and~$i_2$.

For our triangulation of~$M$, we use the same Definition~\ref{d:f} for g-vectors~$e_b$, understanding $p_{v,\hat i}$ of course as the determinant made of the rows~$\mathcal N_j$~\eqref{Ni} corresponding to all vertices of~$v$ except~$j=i$, taken in the increasing order of numbers~$j$. A g-coloring is, in analogy with Definition~\ref{d:pc}, any element of the linear space spanned by all~$e_b$.

\begin{definition}\label{d:pM}
A coloring of a triangulation of~$M$ is \emph{permitted} provided its restriction on any d-simplex coincides with the restriction of some g-coloring.
\end{definition}

Below, $\mathcal F$ means the field of rational functions of all~$\mu_{i,j}$ present in the rows~\eqref{Ni} corresponding to the vertices of the considered triangulation; these~$\mu_{i,j}$ are indeterminates over field $F=\mathbb Q$ or~$\mathbb F_p$. Moreover, if a new vertex arises as a result of a Pachner move, we tacitly extend~$\mathcal F$ so as to include the rational functions of new~$\mu$'s as well.

Let there be a polygon cocycle over $\mathcal F$, consisting of mappings $\mathfrak c_{\Delta^{2n-1}}$ for all $\Delta^{2n-1}\subset M$. Then $\mathfrak c_{\Delta^{2n-1}}(r)$ considered as a function of~$\Delta^{2n-1}$ is a usual cocycle---just because the coboundary~\eqref{cb}, for a fixed permitted coloring~$r$, is nothing but the usual simplicial coboundary.

For a given~$r$, consider the value
\begin{equation}\label{cM}
\sum _{\mathrm{all\;}\Delta^{2n-1}\subset M} \mathfrak c_{\Delta^{2n-1}}(r)
\end{equation}
Expression~\eqref{cM} is a polynomial function of a permitted coloring~$r$ ---that is, of the coordinates of vector~$r$ in the linear space~$V_p$ of permitted colorings w.r.t.\ some basis. 

\begin{definition}\label{d:i}
We define $I(M)$ as expression~\eqref{cM} considered as a function of the coordinates of vector~$r$ and taken up to an $\mathcal F$-linear change of these coordinates and to a `stabilization'---adding more coordinates on which \eqref{cM} does not depend, or removing such coordinates.
\end{definition}

The dimension of~$V_p$ can change under Pachner moves, so there must be no surprise that we define~$I(M)$ in such a `stable' way.

\subsection{Invariance theorem}\label{ss:ip}

\begin{theorem}\label{th:ip}
$I(M)$ is an invariant of PL manifold~$M$.
\end{theorem}

In particular, $I(M)$ does not depend on the indeterminates entering matrix~$\mathcal N$~\eqref{N}.

\begin{proof}
We must show that $I(M)$ does not change under a Pachner move. Recall that a Pachner move replaces a cluster~$C_{\mathrm{ini}}$ of d-simplices with another cluster~$C_{\mathrm{fin}}$ in such way that these clusters form together~$\partial\Delta^{2n}$---the boundary of a P-simplex.

First, we can extend any permitted coloring onto~$C_{\mathrm{fin}}$ (glued to~$M$ by the boundary $\partial C_{\mathrm{fin}} = \partial C_{\mathrm{ini}}$). This can be seen from~\eqref{gog} and the fact that permitted colorings are, for a single d-simplex, the same as g-colorings. Hence, permitted colorings are parameterized by the input vector of the corresponding side of~\eqref{gog}, and these vectors are either the same or one of them is a direct sum of the other with the vector of `superfluous inputs', according to what we remarked after Theorem~\ref{th:F}---in any case, there is no problem to either remove or add these `superfluous inputs'.

Second, we are dealing with a polygon cocycle, hence
\begin{equation}\label{wD}
\pm\! \sum _{\Delta^{2n-1}\subset \Delta^{2n}}\! \mathfrak c_{\Delta^{2n-1}}(r)\; = \sum _{\Delta^{2n-1}\subset C_{\mathrm{ini}}}\! \mathfrak c_{\Delta^{2n-1}}(r)\; - \! \sum _{\Delta^{2n-1}\subset C_{\mathrm{fin}}}\! \mathfrak c_{\Delta^{2n-1}}(r) = 0.
\end{equation}
The minus sign in the middle part of~\eqref{wD} is due to the fact that the mutual orientation of the two clusters induced by an orientation of~$\Delta^{2n}$ is \emph{opposite} to their mutual orientation in the situation when one of them replaces the other within a triangulation of~$M$. Hence, expression~\eqref{cM} remains the same under a Pachner move.

Third, take for coordinates of vector~$r$ the inputs of our general odd-gon relation~\eqref{gog} corresponding to this move, plus the colors of some other faces linearly independent of these inputs (thus lying outside the move). One side of~\eqref{gog} may differ from the other by some `free' inputs (coinciding, due to the same relation~\eqref{gog}, with their outputs) not influencing~\eqref{cM}. This is exactly what is required for~\eqref{cM} to remain the same function of coordinates of~$r$ up to their linear change and `stabilization'.
\end{proof}

\subsection{Invariant on the factor space of permitted colorings modulo g-colorings}\label{ss:f}

In this subsection, we work in the bipolynomial setting, that is, with a double coloring $r=(r_1,r_2)$.

\begin{proposition}\label{p:t}
In a bipolynomial case, \eqref{cM} does not change if one adds a g-coloring to either\/~$r_1$ or~$r_2$.
\end{proposition}

\begin{proof}
Enough to consider the g-coloring generated by \emph{one} g-simplex~$b$. Cocycle $\mathfrak c_{\Delta^{2n-1}}(r_1,r_2)$ changes only locally, in a neighborhood of~$b$ which is topologically just a ball. Hence, this change is a coboundary, and adds nothing to~\eqref{cM}.
\end{proof}

In view of Proposition~\ref{p:t}, it is natural to consider the factor space
\begin{equation}\label{V}
V = V_p / V_g,
\end{equation}
where $V_p$ and~$V_g$ mean linear spaces of permitted and g-colorings, respectively.

\begin{proposition}\label{p:V}
The dimension of factor space~$V$ is a manifold invariant.
\end{proposition}

\begin{proof}
As we remarked after Theorem~\ref{th:F}, the l.h.s.\ of~\eqref{gog} has the same inputs and outputs as the r.h.s.\ plus, generally, some more inputs coinciding with their corresponding outputs. For the l.h.s., coordinates in~$V_p$ can be chosen the same as in the r.h.s.\ plus the values on those latter inputs.

The same applies to coordinates in~$V_g$, but there is one thing here that must be checked: that all the mentioned additional inputs can be made arbitrary \emph{by choosing a proper linear combination of g-vectors~$e_b$}, and this must be done independently of the other coordinates. There are $\frac{k(k+1)}{2}$ of these additional inputs (\,= outputs) for move $(n+1+k) \mapsto (n-k)$, and they belong to \emph{inner} faces of the initial d-simplex cluster, while all the other inputs and outputs belong to the boundary faces. So, it is enough to show that the g-vectors corresponding to \emph{inner g-simplices~$b$} generate a $\frac{k(k+1)}{2}$-dimen\-sional space of colorings---all being, of course, colorings of only inner faces.

The d-simplices in the initial cluster form the star of an inner $(n-k-1)$-simplex; denote this latter as~$\sigma$. Consider the simplex having all vertices of~$\sigma$ plus $(k+1)$ other vertices of the initial cluster, these latter chosen arbitrarily; this will be an $n$-simplex, denote it~$\tau$. According to Proposition~\ref{p:n}, vectors~$e_b$ for \emph{all} $b\subset \tau$ are linearly independent, and so are, then, vectors~$e_b$ for $\sigma \subset b\subset \tau$. As one can see, all such~$b$ are, first, inner g-simplices, and second, there are exactly $\frac{k(k+1)}{2}$ of them.
\end{proof}

Factor space~$V$~\eqref{V} can indeed be nontrivial, as the two following examples show that are checked by direct calculations.

\begin{example}\label{x:V3}
For $M=S^1 \times S^2$ and $F=\mathbb Q$, or~$\mathbb F_2$, or~$\mathbb F_3$, or~$\mathbb F_5$, the dimension of linear space~$V$ is two.
\end{example}

\begin{example}\label{x:V5}
For $M=S^2 \times S^3$ and $F=\mathbb Q$, or~$\mathbb F_2$, or~$\mathbb F_3$, or~$\mathbb F_5$, the dimension of linear space~$V$ is six.
\end{example}

\begin{proposition}\label{p:fa}
Invariant~$I(M)$ depends only on the equivalence classes of two permitted colorings modulo g-colorings, that is, of two elements of the factor space~$V$.
\end{proposition}

\begin{proof}
Follows from Proposition~\ref{p:t}.
\end{proof}

Hence, we can consider~$I(M)$ as a bipolynomial on~$V$, taken up to linear automorphisms of~$V$.

\section[Symmetric bilinear $(2n-2)$-cocycle and how it yields bipolynomial $(2n-1)$-cocycles]{Symmetric bilinear $\boldsymbol{(2n-2)}$-cocycle and how it yields bipolynomial $\boldsymbol{(2n-1)}$-cocycles}\label{s:2n-2}

\subsection[Symmetric bilinear $(2n-2)$-cocycle]{Symmetric bilinear $\boldsymbol{(2n-2)}$-cocycle}\label{ss:bc}

In this section, we are working within the simplicial complex $K=\Delta^{2n}$, with vertices $1,\ldots,2n+1$, and even within its boundary~$\partial\Delta^{2n}$ which consists of $2n+1$ simplices of dimension~$(2n-1)$ and is the union of the two parts of a $(2n-1)$-dimen\-sional Pachner move.

Let there be two colorings of~$K$, and let $x_{iq}$ and~$y_{iq}$ denote the corresponding colors of face~$iq$. Recall that $iq$ means, in this context, the $(2n-2)$-dimen\-sional face containing all vertices $1,\ldots,2n+1$ \emph{except} $i$ and~$q$. Its \emph{orientation} will be important now, so, by definition, if $i<q$, then the orientation of~$iq$ is determined by the increasing order of its vertices, and is opposite to that if $q<i$.

Similarly, orientation of a d-simplex~$q$ will also be determined by the increasing order of its vertices.

The following symbol~$\epsilon _i^{(q)}=\pm 1$ is the incidence number between oriented cells $q$ and~$iq$:
\begin{equation}\label{eps_ip}
\epsilon _i^{(q)} \stackrel{\mathrm{def}}{=} \begin{cases} (-1)^{i} & \text{if}\:\; i<q \\
                                                           -(-1)^{i} & \text{if}\;\; i>q \end{cases}
\end{equation}
Additionally, one more symbol~$\eta _i^{(p)}=\pm 1$ will be of use:
\begin{equation}\label{eta_ip}
\eta _i^{(q)} \stackrel{\mathrm{def}}{=} \begin{cases} (-1)^{i-q} & \text{if}\:\; i<q \\
                                                           -(-1)^{i-q} & \text{if}\;\; i>q \end{cases}
\end{equation}

\begin{proposition}\label{p:2n-2}
The cochain consisting of mappings (compare~\eqref{bc})
\begin{equation}\label{2n-2}
(x_{iq}, y_{iq}) \mapsto c_{iq\,} x_{iq} y_{iq}
\end{equation}
for each face~$iq$, where
\begin{equation}\label{cip}
c_{iq} = \frac{\eta _i^{(q)}}{\prod_{j \ne i,q} d_{j iq}},
\end{equation}
is a bilinear $(2n-2)$-cocycle.
\end{proposition}

Note that \eqref{cip} yields $c_{iq}=c_{qi}$, because both the numerator and denominator change their signs under $i \leftrightarrow q$; the denominator because it is the product going over the $(2n-1)$ vertices---an odd number---of face~$iq$.

\medskip

We introduce the following \emph{scalar product} of two arbitrary colorings $x$ and~$y$ of a $(2n-1)$-simplex~$q$:
\begin{equation}\label{xy_2n-2}
\langle x, y \rangle_{2n-2}^{(q)} = \sum _{i\ne q} \epsilon_i^{(q)} c_{iq}\,x_{iq}y_{iq} \, .
\end{equation}
For \emph{permitted} $x$ and~$y$ this is, of course, just the coboundary of~\eqref{2n-2} for $(2n-1)$-simplex~$q$, hence, what we must prove is that \eqref{xy_2n-2} vanishes for permitted $x$ and~$y$.

Pay attention to the subscript~$2n-2$ in the l.h.s.\ of~\eqref{xy_2n-2}: it emphasizes that we are dealing with $(2n-2)$-cochains, and serves to distinguish this scalar product from another one introduced below in~\eqref{q5}.

\begin{proof}[Proof of Proposition~\ref{p:2n-2}]
Take, first, $(n-2)$-simplex vectors for two \emph{non-inter\-secting} simplices $\Delta_1^{n-2}$ and~$\Delta_2^{n-2}$ as $x$ and~$y$. There are just two faces of~$q$ containing both $\Delta_1^{n-2}$ and~$\Delta_2^{n-2}$, hence, there are two nonvanishing summands in the l.h.s.\ of~\eqref{xy_2n-2}, and one can see, using explicit expressions \eqref{cip} and~\eqref{e_b}, that they cancel each other.

Then, it is not hard to see that if we use a suitable chain of three-term linear dependences~\eqref{lr_2n-1}, we can show that \eqref{xy_2n-2} vanishes for \emph{any} $(n-2)$-simplex vectors $x$ and~$y$ and hence for any permitted $x$ and~$y$.
\end{proof}

\begin{proposition}\label{p:nm}
There are no $(2n-2)$-cocycles linearly independent from cocycle given by \eqref{2n-2} and~\eqref{cip}.
\end{proposition}

\begin{proof}
Consider again, as in the proof of Proposition~\ref{p:2n-2}, $(n-2)$-simplex vectors for two non-inter\-secting simplices $\Delta_1^{n-2}$ and~$\Delta_2^{n-2}$ as $x$ and~$y$ in~\eqref{xy_2n-2}. As there are just two nonvanishing summands in the r.h.s.\ of~\eqref{xy_2n-2}, and they give zero together, the ratio of the two corresponding coefficients~$c_{iq}$ is determined uniquely. Clearly, there is a chain $c_{iq}:c_{jq}:\,\cdots\,:c_{mq}$ of such ratios connecting any given $c_{iq}$ and~$c_{mq}$.
\end{proof}

\subsection[Bipolynomial $(2n-1)$-cocycles in finite characteristic from the bilinear $(2n-2)$-cocycle in characteristic~$0$]{Bipolynomial $\boldsymbol{(2n-1)}$-cocycles in finite characteristic from the bilinear $\boldsymbol{(2n-2)}$-cocycle in characteristic~$\boldsymbol{0}$}\label{ss:pc}

It turns out that our bilinear $(2n-2)$-cocycle~\eqref{2n-2},~\eqref{cip} can yield also \emph{bipolynomial $(2n-1)$-cocycles} needed for constructing an invariant such as described in our Section~\ref{s:f}. To be more exact, we take our bilinear $(2n-2)$-cocycle~\eqref{2n-2},~\eqref{cip} over the field~$\mathcal F$ of rational functions of all needed indeterminates living in vertices over field~$F=\mathbb Q$ of rational numbers. Let $(x,y)$ denote a pair of permitted colorings, $(x_v,y_v)$ its components on a $(2n-2)$-face~$v$, and write the corresponding mapping~\eqref{bc} as
\begin{equation}\label{fb}
\mathfrak c_v\colon \;\, (x_v,y_v) \mapsto c_v x_v y_v,
\end{equation}
changing notations just a bit with respect to~\eqref{2n-2}.

Let $p$ be a prime number, $k$ a natural number, and consider the following five consecutive operations on cocycle~\eqref{fb}.
\begin{enumerate}\itemsep 0pt
 \item\label{i:r} Raise each expression~$c_v x_v y_v$ to the power~$p^k$. Note that although this may call to mind a Frobenius endomorphism, we are, at the moment, in characteristic~$0$. We get a bipolynomial $(2n-2)$-\emph{cochain} with components
\begin{equation}\label{k}
c_v^{p^k} x_v^{p^k} y_v^{p^k}.
\end{equation}
 \item\label{i:b} Take the coboundary of cochain~\eqref{k}.
 \item\label{i:d} Divide the result by~$p$.
 \item\label{i:e} Reduce the result modulo~$p$.
 \item\label{i:t} Optional restriction: set $y=x$.
\end{enumerate}

More rigorously, item~\ref{i:e} means the following. First, we introduce a \emph{discrete valuation}~\cite{discr-val} on field $\mathbb Q(\alpha _1, \ldots, \gamma _{2n+1})$ as follows. Let $P(\alpha _1, \ldots, \gamma _{2n+1})$ and~$Q(\alpha _1, \ldots, \gamma _{2n+1})$ be polynomials with \emph{integer coefficients} and such that there is there is at least one coefficient not divisible by~$p$ in both $P$ and~$Q$, then the valuation is
\begin{equation}\label{pl}
p^l \, \frac{P(\alpha _1, \ldots, \gamma _{2n+1})}{Q(\alpha _1, \ldots, \gamma _{2n+1})}\, \mapsto\, l.
\end{equation}
Then we extend this valuation to polynomials whose variables are face colors $x_v$ and~$y_v$ (for all $(2n-2)$-faces in our P-simplex) and coefficients lie in $\mathbb Q(\alpha _1, \ldots, \gamma _{2n+1})$, setting the valuation of all $x_v$ and~$y_v$ to zero. Note that this entails no contradiction, because of the explicit form of linear dependences between colors given by matrices described in Subsection~\ref{ss:mat}, with entries~\eqref{A-expl}.

Reduction modulo~$p$ of the l.h.s.\ of~\eqref{pl} can be done if $l\ge 0$, and gives zero for $l>0$ and
\begin{equation*}
\frac{P(\alpha _1, \ldots, \gamma _{2n+1})}{Q(\alpha _1, \ldots, \gamma _{2n+1})}
\end{equation*}
for $l=0$, where we take the liberty of denoting the polynomials and indeterminates over~$\mathbb F_p$ by the same letters as over~$\mathbb Q$.

\begin{proposition}\label{p:f}
The first four steps \ref{i:r}--\ref{i:e} lead correctly from bilinear cocycle~\eqref{fb} to a bipolynomial $(2n-1)$-cocycle of degrees $p^k$ in both $x$ and~$y$ over the field~$\mathcal F_p = \mathbb F_p (\alpha _1, \ldots, \gamma _{2n+1})$ (compare~\eqref{bF}), where $\mathbb F_p$ is of course the prime field of $p$ elements. 

The fifth optional step~\ref{i:t} gives a polynomial $(2n-1)$-cocycle of degree~$2p^k$.
\end{proposition}

\begin{proof}
We need to prove the following two points.

\textit{Feasibility of item~\ref{i:e}.}
Let $w$ be an oriented $(2n-1)$-simplex, and $\epsilon _v^{(w)}$ the incidence number between~$w$ and its face~$v$. As~\eqref{fb} is a cocycle,
\begin{equation}\label{cvxvyv}
\sum _{v\subset w} \epsilon _v^{(w)} c_v x_v y_v = 0.
\end{equation}
Our discrete valuation for all summands in~\eqref{cvxvyv} is zero (taking into account~\eqref{cip} for coefficients~$c_v$). Using~\eqref{cvxvyv}, we can express one of the summands through the rest of them and substitute in
\begin{equation}\label{cxyk}
\sum _{v\subset w} \epsilon _v^{(w)} (c_v x_v y_v)^{p^k}.
\end{equation}
When we expand the result, its coefficients will clearly be all divisible by~$p$.

\textit{The result of~\ref{i:e} is a cocycle.}
This follows from the fact that the result of~\ref{i:d} is a coboundary and hence a cocycle.

\end{proof}

A handy explicit expression for the resulting $(2n-1)$-cocycle can be obtained using \emph{Newton's identities}~\cite[Section~I.2]{Macdonald} for power sums and elementary symmetric functions, with the summands $\epsilon _v^{(w)} c_v x_v y_v$ in the l.h.s.\ of~\eqref{cvxvyv} taken as variables. Equality~\eqref{cvxvyv} means of course that their \emph{first power sum vanishes}.

In characteristic two, the following symbol is useful in these calculations and writing out the results:
\begin{equation*}
\tilde{\epsilon}_v^{(w)} \,\stackrel{\mathrm{def}}{=}\, \frac{\epsilon _v^{(w)}+1}{2} = 
\begin{cases} 1 & \text{if \ } \epsilon _v^{(w)}=1 \\
              0 & \text{if \ } \epsilon _v^{(w)}=-1 \end{cases}
\end{equation*}

\begin{example}\label{x:h2}
For $p=2$, \ $k=1$, the $(2n-1)$-cocycle value on a d-sim\-plex~$w$ is
\begin{equation}\label{h2}
\sum _{\substack{v,v'\subset w\\ v < v'}} c_v c_{v'} x_v x_{v'}y_v y_{v'}
 +\sum _{v\subset w} \tilde{\epsilon}_v^{(w)} (c_v x_v y_v)^2,
\end{equation}
where $v$ and~$v'$ are its faces. Taking some liberty, we also assume that these faces are numbered, and understand their \emph{numbers} when writing ``$v < v'$''.

Our calculations for specific manifolds below in Subsection~\ref{ss:M} show that \eqref{h2} is, at least in the \emph{pentagon} ($n=2$) and \emph{heptagon} ($n=3$) cases, a \emph{nontrivial} cocycle---not a coboundary.
\end{example}

Note, by the way, that in characteristic two, a Frobenius endomorphism applied to~\eqref{cvxvyv} gives
\begin{equation*}
\sum _{v\subset w} (c_v x_v y_v)^2 = 0.
\end{equation*}
Hence, $\tilde{\epsilon}_v^{(w)}$ in~\eqref{h2} can be replaced painlessly by $1-\tilde{\epsilon}_v^{(w)}$, if needed.

\begin{example}\label{x:h3}
For $p=3$, \ $k=1$, the $(2n-1)$-cocycle value on a d-sim\-plex~$w$ is
\begin{equation}\label{h3}
\sum _{\substack{v_1,v_2,v_3\subset w\\ v_1 < v_2 < v_3}} \epsilon _{v_1}^{(w)} \epsilon _{v_2}^{(w)} \epsilon _{v_3}^{(w)}\, c_{v_1}c_{v_2}c_{v_3}\, x_{v_1}x_{v_2}x_{v_3}\, y_{v_1}y_{v_2}y_{v_3},
\end{equation}
where $v_1$, $v_2$ and~$v_3$ are its faces.

A calculation shows that \eqref{h3} is, at least in the heptagon case, again a nontrivial cocycle and, moreover, remains nontrivial after the restriction $y=x$ of our item~\ref{i:t}.
\end{example}

\subsection{Calculations for specific manifolds and characteristic two}\label{ss:M}

Here we show by direct calculations that at least cocycle~\eqref{h2}, over a field of characteristic~2, gives rise to a nontrivial invariant.

In all studied cases, $I(M)$ turned out to be a ``bi-semilinear form'' with respect to \emph{squaring} the coordinates in the factor space~$V=V_p/V_g$. That is, if we denote $X$ and~$Y$ the rows of squares of coordinates of two permitted colorings $x$ and~$y$ of~$M$ w.r.t.\ some basis in~$V$, then $I(M)$ is represented by a symmetric bilinear form~$B(X,Y)$ such that $B(X,X)=0$ identically. This all does not seem very evident from the cocycle expression~\eqref{h2}.

The calculation results are in Table~\ref{t:hm}. Note that in characteristic~2 the mentioned bilinear forms~$B(X,Y)$ are also \emph{symplectic} and can be brought by a linear transformation of their variables to a canonical block diagonal form with some blocks $\begin{pmatrix} 0 & 1 \\ 1 & 0 \end{pmatrix}$ and the other zeroes, see for instance~\cite{Grove}. Note also that, in characteristic~2, a linear transformation of `quadratic' vectors $X$ and~$Y$ corresponds to a basis change in space~$V$. Hence, the obvious invariant of such a form is its \emph{rank}, so we present these ranks in our Table~\ref{t:hm}.
\begin{table}
 \centering
 \begin{subtable}[b]{0.4\textwidth}
  \begin{equation*}
   \renewcommand{\arraystretch}{1.5}
   \begin{array}{|c|c|c|}
    \hline
    M & \dim V & \rank B \\ \hline
    \hline
    L(3,1) & 0 & 0 \\ \hline
    L(4,1) & 2 & 0 \\ \hline
    T^3 & 6 & 0 \\ \hline
    S^1\times S^2 & 2 & 0 \\ \hline
    S^1\times \mathbb RP^2 & 4 & 4 \\ \hline
    \mathbb RP^3 & 2 & 0 \\ \hline
   \end{array}
  \end{equation*}
  \caption{Three-dimensional manifolds}
  \label{t:c2d3}
 \end{subtable}  
  \qquad
 \begin{subtable}[b]{0.5\textwidth}
  \begin{equation*}
   \renewcommand{\arraystretch}{1.5}
   \begin{array}{|c|c|c|}
    \hline
    M & \dim V & \rank B \\ \hline
    \hline
    S^3\times \mathbb RP^2 & 6 & 0 \\ \hline
    S^2\times S^3 & 6 & 0 \\ \hline
    S^2\times \mathbb RP^3 & 12 & 0 \\ \hline
    S^1\times \mathbb RP^4 & 12 & 12 \\ \hline
    S^1\times \mathbb RP^2\times \mathbb RP^2 & 30 & 24 \\ \hline
    \mathbb RP^2\times \mathbb RP^3 & 18 & 12 \\ \hline
    T^5 & 60 & 0 \\ \hline
   \end{array}
  \end{equation*}
  \caption{Five-dimensional manifolds}
  \label{t:c2d5}
 \end{subtable}  
 \caption{Invariants following from $I(M)$ for some specific manifolds}
 \label{t:hm}
\end{table}

\subsection[More nontrivial $(2n-1)$-cocycles for one P-simplex]{More nontrivial $\boldsymbol{(2n-1)}$-cocycles for one P-simplex}\label{ss:oc}

Computer calculations of the dimension of a $(2n-1)$-cohomology space for a P-simplex show that it can be surprisingly nontrivial. Table~\ref{t:hn} presents some calculation results waiting for their theoretical explanation.
\begin{table}
 \centering
 \begin{subtable}[b]{0.45\textwidth}
  \begin{equation*}
   \renewcommand{\arraystretch}{1.5}
   \begin{array}{|c|c|c|c|}
    \hline
    \characteristic \mathcal F & \; n \; & \text{degree} & \dim H^{2n-1} \\ \hline
    \hline
    0 & 3 & 2 & 1 \\ \hline
    0 & 3 & 3 & 0 \\ \hline
    0 & 3 & 4 & 0 \\ \hline
    2 & 3 & 2 & 6 \\ \hline
    2 & 3 & 3 & 6 \\ \hline
    2 & 3 & 4 & 7 \\ \hline
    3 & 3 & 3 & 6 \\ \hline
    3 & 3 & 4 & 6 \\ \hline
    3 & 3 & 5 & 0 \\ \hline
    3 & 3 & 6 & 2 \\ \hline
    5 & 3 & 5 & 6 \\ \hline
    5 & 3 & 6 & 6 \\ \hline
    7 & 3 & 7 & 6 \\ \hline
    7 & 3 & 8 & 6 \\ \hline
    11 & 3 & 11 & 6 \\ \hline
    11 & 3 & 12 & 6 \\ \hline
   \end{array}
  \end{equation*}
  \caption{Heptagon case ($n=3$)}
  \label{t:hn-left}
 \end{subtable}
  \qquad
 \begin{subtable}[b]{0.45\textwidth}
  \begin{equation*}
   \renewcommand{\arraystretch}{1.5}
   \begin{array}{|c|c|c|c|}
    \hline
    \characteristic \mathcal F & \; n \; & \text{degree} & \dim H^{2n-1} \\ \hline
    \hline
    0 & 4 & 2 & 0 \\ \hline
    2 & 4 & 2 & 10 \\ \hline
    2 & 4 & 3 & 8 \\ \hline
    2 & 4 & 4 & 11 \\ \hline
    3 & 4 & 3 & 10 \\ \hline
    3 & 4 & 4 & 8 \\ \hline
    3 & 4 & 5 & 0 \\ \hline
    5 & 4 & 5 & 10 \\ \hline
    5 & 4 & 6 & 8 \\ \hline
   \end{array}
  \end{equation*}
  \caption{Enneagon case ($n=4$)}
  \label{t:hn-right}
 \end{subtable}
  \caption{Dimensions of some spaces of nontrivial $(2n-1)$-cocycles for~$\Delta^{2n}$. Recall that $n=3$ corresponds to heptagon, and $n=4$ ---to enneagon}
 \label{t:hn}
\end{table}

The P-simplex case looks important because the boundary of a P-simplex is where a Pachner move takes place. It must be remarked, however, that in order to be applied to arbitrary manifolds, there must be, first, an explicit expression for a $(2n-1)$-cocycle, and second, this expression must have such form that can be applied not only to the d-simplices taking part in a Pachner move but also to the d-simplices in an arbitrary triangulation. Formulas \eqref{h2} or~\eqref{h3} are good examples of such expressions.

Although there are still problems to be solved on this way, we present below in Section~\ref{s:5c} some interesting calculations for one particular case (the first line in Table~\ref{t:hn-left}).

\section{Heptagon: nontrivial bilinear 5-cocycle in characteristic zero}\label{s:5c}

In Section~\ref{s:2n-2}, we showed how to construct polynomial $(2n-1)$-cocycles \emph{in finite characteristic} from the bilinear $(2n-2)$-cocycle~\eqref{2n-2},~\eqref{cip}. Very interestingly, in the \emph{heptagon} case, $n=3$, a nontrivial \emph{quadratic} $(2n-1)=5$-cocycle exists already in characteristic~0, for the simplicial complex $K=\Delta^6$ ---that is, what we call a P-simplex for the heptagon relation. Of course, we can obtain a symmetric bilinear cocycle by polarizing this quadratic cocycle.

\subsection{Why a nontrivial quadratic 5-cocycle must exist}\label{ss:5c-why}

We first calculate the numbers of linearly independent quadratic 4-, 5- and 6-cochains in our simplicial complex $K=\Delta^6=1234567$.

\paragraph{4-cochains.}
There are 21 linearly independent 4-cochains, one for each 4-dimen\-sional face~$iq$, namely cochains~$x_{iq}^2$, where $x_{iq}$ is the color of~$iq$. Recall that ``$iq$'' means the face containing all vertices $1,\ldots,7$ \emph{except} $i$ and~$q$.

\paragraph{5-cochains.}
For each 5-simplex, the linear space of permitted colorings is 3-dimen\-sional: three arbitrary ``input'' colors determine three ``output'' ones. The space of quadratic forms of 3 variables is 6-dimen\-sional. And there are seven 5-simplices (that is, simplices with \emph{six} vertices!) in the heptagon relation.

Hence, there is the $7\times 6=42$-dimen\-sional linear space consisting of 7-tuples of quadratic forms, one for each 5-simplex.

\paragraph{6-cochains.}
There are six independent ``input'' colors for the whole heptagon (see again Figure~\ref{fig:hepta-new}), so there are $\frac{6\times 7}{2}=21$ linearly independent quadratic forms.

\medskip

We now write out a fragment of sequence~\eqref{hcc} for quadratic cochains, \emph{assuming that the characteristic of our field~$F$ is zero}:
\begin{equation}\label{456}
\begin{pmatrix} 21\\ \text{4-cochains} \end{pmatrix} \xrightarrow[\mathrm{rank} = 20]{\textstyle\delta} \begin{pmatrix} 42\\ \text{5-cochains} \end{pmatrix} \xrightarrow[\mathrm{rank} = 21]{\textstyle\delta} \begin{pmatrix} 21\\ \text{6-cochains} \end{pmatrix} .
\end{equation}
Here follow the explanations.

First, ``21 cochains'' stays in~\eqref{456} for ``21-dimen\-sional space of cochains'', and so on.

Second, the rank of the left coboundary operator~$\delta$ in~\eqref{456} is $20$ and not~$21$ because there exists exactly one-dimen\-sional space of 4-cocycles, according to Propositions \ref{p:2n-2} and~\ref{p:nm}.

Third, the rank of the \emph{right} operator~$\delta$ is surely $\le 21$, and this is already enough to conclude that the cohomology space dimension in the middle term is $\ge 42-20-21=1$. Actually, a direct calculation, made in characteristic~$0$, shows that the rank of the right~$\delta$ is exactly~$21$ for generic matrices~$\mathcal M$~\eqref{A}, but is looks, at this moment, more difficult to understand why it is so.

\subsection{Explicit form of 5-cocycle}\label{ss:xpl5}

We define now one more scalar product, this time between two \emph{permitted} colorings of a 5-simplex~$q$. These are generated by \emph{edge vectors}, because g-simplices are edges for heptagon. For the case where these colorings are the restrictions of edge vectors $e_{ij}$ and $e_{kl}$, respectively, on~$q$, we set
\begin{equation}\label{q5}
\langle e_{ij}, e_{kl} \rangle_5^{(q)} \, \stackrel{\mathrm {def}}{=} \, \det\eta_q \cdot (d_{ikq}d_{jlq}+d_{ilq}d_{jkq}),
\end{equation}
where
\begin{equation}\label{eta}
\eta_p = 
\begin{pmatrix} \alpha_i^2 & \alpha_j^2 & \alpha_k^2 & \alpha_l^2 & \alpha_m^2 & \alpha_n^2 \\[.5ex]
                \beta_i^2 & \beta_j^2 & \beta_k^2 & \beta_l^2 & \beta_m^2 & \beta_n^2 \\[.5ex]
                \gamma_i^2 & \gamma_j^2 & \gamma_k^2 & \gamma_l^2 & \gamma_m^2 & \gamma_n^2 \\[.5ex]
                \alpha_i \beta_i & \alpha_j \beta_j & \alpha_k \beta_k & \alpha_l \beta_l & \alpha_m \beta_m & \alpha_n \beta_n \\[.5ex]
                 \alpha_i \gamma_i & \alpha_j \gamma_j & \alpha_k \gamma_k & \alpha_l \gamma_l & \alpha_m \gamma_m & \alpha_n \gamma_n \\[.5ex]
                 \beta_i \gamma_i & \beta_j \gamma_j & \beta_k \gamma_k & \beta_l \gamma_l & \beta_m \gamma_m & \beta_n \gamma_n
\end{pmatrix} ,
\end{equation}
and $i,\ldots,n$ are the numbers from 1 through~7 \emph{except}~$q$, going in the increasing order.

As edge vectors make not a basis but an \emph{overfull system} of vectors in the linear space of all permitted colorings, the following proposition is necessary to justify this definition.

\begin{proposition}\label{p:corr5}
Formula~\eqref{q5} defines a scalar product in the linear space of permitted colorings of 5-simplex~$q$ correctly.
\end{proposition}

\begin{proof}
It must be checked that our definition~\eqref{q5} agrees with three-term linear dependences~\eqref{lr_2n-1}. For instance, there is the following linear dependence between three versions of vector~$e_{ij}$:
\begin{equation}\label{cor1}
d_{j_2j_3q} e_{ij_1}|_q - d_{j_1j_3q} e_{ij_2}|_q + d_{j_1j_2q} e_{ij_3}|_q = 0,
\end{equation}
so, the sum of the three corresponding scalar products~\eqref{q5} must give zero. And this is indeed so due to the determinants $d_{jlq}$ and~$d_{jkq}$ entering in the r.h.s.\ of~\eqref{q5}. For instance, for the first of these determinants, a simple consequence from Pl\"ucker bilinear relation gives
\begin{equation}\label{cor2}
d_{j_2j_3q} d_{j_1 lq} - d_{j_1j_3q} d_{j_2 lq} + d_{j_1j_2q} d_{j_3 lq} = 0.
\end{equation}
Hence, for the scalar product defined according to~\eqref{q5} we also have the desirable equality showing that definition~\eqref{q5} is self-consistent:
\begin{equation}\label{corr}
d_{j_2j_3q} \langle e_{ij_1}, e_{kl} \rangle_5^{(q)} - d_{j_1j_3q} \langle e_{ij_2}, e_{kl} \rangle_5^{(q)} + d_{j_1j_2q} \langle e_{ij_3}, e_{kl} \rangle_5^{(q)} = 0. 
\end{equation}
\end{proof}

\begin{proposition}\label{p:cr}
Formulas \eqref{q5} and~\eqref{eta} define, indeed, a cocycle:
\begin{equation}\label{5c-usl}
\sum _{q=1}^7 (-1)^q \langle e_{ij}, e_{kl} \rangle_5^{(q)} = 0
\end{equation}
for any two edges $ij$ and~$kl$.
\end{proposition}

\begin{proof}
Direct calculation.
\end{proof}

\subsection{Nontriviality}\label{ss:ntr}

\begin{proposition}\label{p:nt}
Cocycle defined according to \eqref{q5} and~\eqref{eta} is nontrivial---not a coboundary.
\end{proposition}

\begin{proof}
Coboundary is, in this situation, a linear combination of~$x_{iq}^2$ taken over 4-faces. The scalar product $\langle x, y \rangle_5^{(q)}$ corresponding to such a cocycle would then be a linear combination of products~$x_{iq}y_{iq}$. Taking into account that, according to~\eqref{e_b},
\begin{equation*}
e_{ij}|_{ijlm} = d_{ilm} d_{jlm}
\end{equation*}
for heptagon, we see that $\langle e_{12}, e_{34} \rangle_5^{(7)}$, $\langle e_{13}, e_{24} \rangle_5^{(7)}$ and $\langle e_{14}, e_{23} \rangle_5^{(7)}$ would all three coincide---but they are actually all different.
\end{proof}

\subsection{One more observation}\label{ss:ob}

Matrix~$\mathcal M$~\eqref{A} can be reduced, by a linear transformation of its rows, to the form where its three first columns form an identity matrix. Suppose this has been already done, that is,
\begin{equation}\label{id3}
\begin{pmatrix} \alpha_1 & \alpha_2 & \alpha_3 \\
                \beta_1 & \beta_2 & \beta_3 & \\ 
                \gamma_1 & \gamma_2 & \gamma_3 \end{pmatrix} =
\begin{pmatrix} 1 & 0 & 0 \\ 0 & 1 & 0 \\ 0 & 0 & 1 \end{pmatrix},
\end{equation}
and consider the determinant of matrix~$\eta_7$~\eqref{eta} for such~$\mathcal M$. A calculation shows that
\begin{align}\label{dh}
\det \eta_7 = 
 \alpha_4  \beta_4  \alpha_5  \gamma_5  \beta_6  \gamma_6 -  \alpha_4  \gamma_4  \alpha_5  \beta_5  \beta_6  \gamma_6 -  \alpha_4  \beta_4  \beta_5  \gamma_5  \alpha_6  \gamma_6 \\
+  \beta_4  \gamma_4  \alpha_5  \beta_5  \alpha_6  \gamma_6 +  \alpha_4  \gamma_4  \beta_5  \gamma_5  \alpha_6  \beta_6 -  \beta_4  \gamma_4  \alpha_5  \gamma_5  \alpha_6  \beta_6 \\
\stackrel{\mathrm {def}}{=} -\dethad
\begin{pmatrix} \alpha_4 & \alpha_5 & \alpha_6 \\
                \beta_4 & \beta_5 & \beta_6 & \\ 
                \gamma_4 & \gamma_5 & \gamma_6 \end{pmatrix} ,
\end{align}
where function `$\dethad$' on $3\times 3$ matrices was introduced in~\cite[Eq.~(7)]{igrushka} in connection with what seemed a completely different problem---evolution of a discrete-time dynamical system, where one step of evolution consisted in taking, first, the usual inverse of a matrix, and second---the ``Hadamard inverse'', that is, inverting each matrix entry separately.

\subsection{Absense of similar cocycles for pentagon, enneagon and hendecagon}\label{ss:a}

We now write out the analogues of sequence~\eqref{456} for pentagon, enneagon and hendecagon, assuming again the zero characteristic for field~$F$.

\paragraph{Pentagon}
\begin{equation}\label{234}
\begin{pmatrix} 10\\ \text{2-cochains} \end{pmatrix} \xrightarrow[\mathrm{rank} = 9]{\textstyle\delta} \begin{pmatrix} 15\\ \text{3-cochains} \end{pmatrix} \xrightarrow[\mathrm{rank} = 6]{\textstyle\delta} \begin{pmatrix} 6\\ \text{4-cochains} \end{pmatrix}
\end{equation}

\paragraph{Enneagon}
\begin{equation}\label{678}
\begin{pmatrix} 36\\ \text{6-cochains} \end{pmatrix} \xrightarrow[\mathrm{rank} = 35]{\textstyle\delta} \begin{pmatrix} 90\\ \text{7-cochains} \end{pmatrix} \xrightarrow[\mathrm{rank} = 55]{\textstyle\delta} \begin{pmatrix} 55\\ \text{8-cochains} \end{pmatrix}
\end{equation}

\paragraph{Hendecagon}
\begin{equation}\label{89A}
\begin{pmatrix} 55\\ \text{8-cochains} \end{pmatrix} \xrightarrow[\mathrm{rank} = 54]{\textstyle\delta} \begin{pmatrix} 165\\ \text{9-cochains} \end{pmatrix} \xrightarrow[\mathrm{rank} = 111]{\textstyle\delta} \begin{pmatrix} 120\\ \text{10-cochains} \end{pmatrix}
\end{equation}

The ranks of the \emph{left} operators~$\delta$ are always less than the number of $(2n-2)$-faces (and $(2n-2)$-cochains) by one, due to the existence of exactly one-dimen\-sional space of $(2n-2)$-cocycles, according to Subsection~\ref{ss:bc}.

The ranks of the \emph{right} operators~$\delta$ are again (like it was for sequence~\eqref{456}) more complicated, and were calculated using computer algebra and for generic parameters (entries of matrices~$\mathcal M$~\eqref{A}).

It follows from these ranks that there are no nontrivial cocycles in the middle terms, in a surprising contrast with the heptagon case!

\section{Discussion}\label{s:d}

Finally, some comments on possible directions of further research.

\paragraph{Polygon relations parameterized by simplicial cocycles.}
The odd-gon relations considered in this paper are, in a sense, the simplest. On the other hand, \emph{simplicial cocycles} arise very naturally in the context of polygon relations, as was shown in~\cite{3--3} on an example of a `fermionic' theory with Grassmannian variables. Further development of the theory has already started~\cite{nonconstant,hepta_simplicial}; it promises to give even richer invariants, namely those of a pair ``manifold, cohomology class''.

\paragraph{From heptagon to hexagon with two-component colors.}
Consider a 4-dimen\-sional Pachner move 3--3 and then the bicones over its l.h.s.\ (initial configuration) and r.h.s.\ (final configuration). Bicone means here the same as the join~\cite[Chapter~0]{Hatcher} with the boundary~$\partial I=\{0,1\}$ of the unit segment $I=[0,1]$. It is an easy exercise to see that the l.h.s.\ bicone can be transformed into the r.h.s.\ one by, first, a 5-dimen\-sional Pachner move 3--4 and, second, move 4--3. If there are now permitted colorings defined for the 4-faces of the 5-simplices involved, like in this paper, then we can attach two colors, or call it a \emph{two-component} color, to each 3-face of 4-simplices in the 3--3 move from which we started. It will be interesting to study connections of this construction with papers~\cite{cubic,nonconstant}.

\paragraph{More polygon cocycles.}
Our calculations in Subsection~\ref{ss:oc} and Section~\ref{s:5c} suggest that there may be many more interesting polygon cocycles.

\end{document}